\documentclass[12pt,onecolumn,journal,draftclsnofoot,singlespace]{IEEEtran}
\makeatletter
\def\ps@headings{%
\def\@oddhead{\mbox{}\scriptsize\rightmark \hfil \thepage}%
\def\@evenhead{\scriptsize\thepage \hfil \leftmark\mbox{}}%
\def\@oddfoot{}%
\def\@evenfoot{}}
\makeatother 

\usepackage{times,amsmath,amssymb,graphicx,epsf,psfrag,epsfig,cite,color}

\def\boxit#1{\vbox{\hrule\hbox{\vrule\kern3pt
        \vbox{\kern3pt#1\kern3pt}\kern3pt\vrule}\hrule}}

\def\reals{ { {\rm  I \kern-0.15em R }  } }
\def\complex{ {\,{{\rm C} \kern-0.50em \raise0.20ex {  |}}\, }}

\def\Rbf{{\bf R}}

\def\Ac{{\cal A}}

\def\Ec{{\cal E}}
\def\Fc{{\cal F}}

\def\be{\begin{equation}}
\def\ee{\end{equation}}

\def\defeq{{\stackrel{\Delta}{=}}}
\def\scalefig#1{\epsfxsize #1\textwidth}
\def\Rxx{\Rbf_{\ssstyle X\kern-.1em X}}

\let\ssstyle=\scriptscriptstyle

\def\etal{{\it et al. \/}}

\def\ie{{\it i.e.,\ \/}}
\def\Kout{\setbox1=\hbox{\Huge\bf K}\hbox to
1.05\wd1{\hspace{.05\wd1}
}}
\def\Sout{\setbox1=\hbox{\Huge\bf S}\hbox to 1.05\wd1{\hspace{.05\wd1}
}}

\def\ie{{\it i.e.,\ \/}}

\def\defeq{{\,\stackrel{\Delta}{=}}\,}

\def\scalefig#1{\epsfxsize #1\textwidth}
\def\nn{{\nonumber}}

\newtheorem{lemma}{Lemma}
\newtheorem{theorem}{Theorem}

\newtheorem{definition}{Definition}

\usepackage{bbm}
\begin{document}

\title{\huge Deterministic Sequencing of Exploration and Exploitation for
Multi-Armed Bandit Problems}

\author{Sattar Vakili,~~~Keqin Liu,~~~Qing Zhao
\thanks{This work was supported by the Army Research Office under Grant
W911NF-12-1-0271. Part of this work was presented at the \emph{Allerton Conference on Communications, Control, and Computing}, September, 2011.}
\thanks{Sattar Vakili, Keqin Liu, and Qing Zhao are with the Department of Electrical and Computer Engineering,
        University of California, Davis, CA, 95616, USA. Email: \{svakili, kqliu, qzhao\}@ucdavis.edu.}
}

\maketitle

\begin{abstract}
In the Multi-Armed Bandit (MAB) problem, there is a given set of
arms with unknown reward models. At each time, a player selects one
arm to play, aiming to maximize the total expected reward over a
horizon of length~$T$. An approach based on a Deterministic
Sequencing of Exploration and Exploitation (DSEE) is developed for
constructing sequential arm selection policies. It is shown that for
all light-tailed reward distributions, DSEE achieves the optimal logarithmic
order of the regret, where regret is defined as the total expected reward loss
against the ideal case with known reward models. For heavy-tailed reward distributions,
DSEE achieves $O(T^{1/p})$ regret when the moments of the reward distributions
exist up to the $p$th order for $1<p\le 2$ and $O(T^{1/(1+p/2)})$ for $p>2$. With the knowledge of an upperbound on a finite moment
of the heavy-tailed reward distributions, DSEE offers the optimal logarithmic regret order. The proposed DSEE approach
complements existing work on MAB by providing corresponding results
for general reward distributions. Furthermore, with a clearly
defined tunable parameter---the cardinality of the exploration
sequence, the DSEE approach is easily extendable to variations of
MAB, including MAB with various
objectives, decentralized MAB with multiple players and incomplete
reward observations under collisions, MAB with unknown Markov dynamics, and combinatorial MAB with dependent
arms that often arise in
network optimization problems such as the shortest path, the minimum spanning, and the dominating set problems
under unknown random weights.
\end{abstract}


\section{Introduction}\label{sec:intro}

\subsection{Multi-Armed Bandit}
\label{sec:intro_MAB}

Multi-armed bandit (MAB) is a class of sequential learning and
decision problems with unknown models. In the classic MAB, there are
$N$ independent arms and a single player. At each time, the player
chooses one arm to play and obtains a random reward drawn i.i.d.
over time from an unknown distribution. Different arms may
have different reward distributions. The design objective is a
sequential arm selection policy that maximizes the total expected
reward over a horizon of length $T$. The MAB problem finds a
wide range of applications including clinical trials, target
tracking, dynamic spectrum access, Internet advertising and Web
search, and social economical networks
(see~\cite{{Robbins:52BAMS},{Santner&Tamhane84book},{MahajanTneketzis07FASM}}
and references therein).

In the MAB problem, each received reward plays two roles: increasing
the wealth of the player, and providing one more observation for
learning the reward statistics of the arm. The tradeoff between
exploration and exploitation is thus clear: which role should be
emphasized in arm selection---an arm less explored thus holding
potentials for the future or an arm with a good history of rewards?
In 1952, Robbins addressed the two-armed bandit
problem~\cite{Robbins:52BAMS}. He showed that the same maximum
average reward achievable under a known model can be obtained by
dedicating two arbitrary sublinear sequences for playing each of the
two arms. In 1985, Lai and Robbins proposed a finer performance
measure, the so-called regret, defined as the expected total reward
loss with respect to the ideal scenario of known reward models
(under which the best arm is always played)~\cite{Lai&Robbins85AAM}.
Regret not only indicates whether the maximum average reward under
known models is achieved, but also measures the convergence rate of
the average reward, or the effectiveness of learning. Although all
policies with sublinear regret achieve the maximum average reward,
the difference in their total expected reward can be arbitrarily
large as $T$ increases. The minimization of the regret is thus of
great interest. Lai and Robbins showed that the minimum regret has a
logarithmic order in $T$ and constructed explicit policies to
achieve the minimum regret growth rate for several
reward distributions including Bernoulli, Poisson, Gaussian,
Laplace~\cite{Lai&Robbins85AAM} under the assumption that the distribution type is known.
In~\cite{AgrawalEtal95AAP}, Agrawal developed simpler index-type policies in explicit form
for the above four distributions as well as exponential distribution assuming known distribution type.
In~\cite{Auer&etal02ML},  Auer~\etal in
2002~\cite{Auer&etal02ML} developed order optimal index policies for any unknown distribution
with bounded support assuming the support range is known.

In these classic policies developed in~\cite{Lai&Robbins85AAM,AgrawalEtal95AAP,Auer&etal02ML}, arms are prioritized according to
two statistics: the sample mean $\bar{\theta}(t)$ calculated from
past observations up to time $t$ and the number $\tau(t)$ of times
that the arm has been played up to $t$. The larger $\bar{\theta}(t)$
is or the smaller  $\tau(t)$ is, the higher the priority given to
this arm in arm selection. The tradeoff between exploration and
exploitation is reflected in how these two statistics are combined
together for arm selection at each given time $t$. This is most
clearly seen in the UCB (Upper Confidence Bound) policy proposed by
Auer~\etal in~\cite{Auer&etal02ML}, in which an index $I(t)$ is
computed for each arm and the arm with the largest index is chosen.
The index has the following simple form:
\begin{eqnarray}\label{eqn:AuerIndex}
I(t)=\bar{\theta}(t)+\sqrt{2\frac{\log t}{\tau(t)}}.
\end{eqnarray}This index form is intuitive in the light of
Lai and Robbins's result on the logarithmic order of the minimum
regret which indicates that each arm needs to be explored on the
order of $\log t$ times. For an arm sampled at a smaller order than
$\log t$ times, its index, dominated by the second term  (referred to as the upper
confidence bound), will be sufficient
large for large $t$ to ensure further exploration.

\subsection{Deterministic Sequencing of Exploration and Exploitation}
\label{sec:intro_DSEE}

In this paper,
we develop a new approach to the MAB problem. Based on a
Deterministic Sequencing of Exploration and Exploitation (DSEE),
this approach differs from the classic policies proposed
in~\cite{Lai&Robbins85AAM,{AgrawalEtal95AAP},{Auer&etal02ML}} by
separating in time the two objectives of exploration and exploitation.
Specifically, time is divided into two interleaving sequences: an exploration sequence
and an exploitation sequence.
In the former, the
player plays all arms in a round-robin fashion. In the latter,
the player plays the arm with the largest sample mean (or a properly chosen mean estimator).
Under this approach, the tradeoff between exploration and
exploitation is reflected in the cardinality of the exploration
sequence. It is not difficult to see that the regret order is lower
bounded by the cardinality of the exploration sequence since a fixed
fraction of the exploration sequence is spent on bad arms.
Nevertheless, the exploration sequence needs to be chosen
sufficiently dense to ensure effective learning of the best arm.
The key issue here is to find the minimum cardinality of the exploration
sequence that ensures a reward loss in the exploitation sequence caused by incorrectly identified arm rank
having an order no larger than the cardinality of the exploration
sequence.

We show that when the reward distributions are light-tailed, DSEE
achieves the optimal logarithmic order of the regret using an
exploration sequence with $O(\log T)$ cardinality. For heavy-tailed reward distributions,
DSEE achieves $O(T^{1/p})$ regret when the moments of the reward distributions
exist up to the $p$th order for $1<p\le 2$ and $O(T^{1/(1+p/2)})$ for $p>2$.
With the knowledge of an upperbound on a finite moment
of the heavy-tailed reward distributions, DSEE offers the optimal logarithmic regret order.

We point out that both the classic policies
in~\cite{Lai&Robbins85AAM,{AgrawalEtal95AAP},{Auer&etal02ML}} and
the DSEE approach developed in this paper require certain knowledge on the
reward distributions for policy construction. The classic policies
in~\cite{Lai&Robbins85AAM,AgrawalEtal95AAP,Auer&etal02ML} apply to specific distributions with either
known distribution types~\cite{Lai&Robbins85AAM,{AgrawalEtal95AAP}} or known finite support range~\cite{Auer&etal02ML}.
The advantage of the DSEE approach is that it applies to any distribution without knowing the distribution type.
The caveat is that it requires the knowledge of a positive lower bound on
the difference in the reward means of the best and the second best
arms. This can be a more demanding requirement than the distribution
type or the support range of the reward distributions. By increasing the
cardinality of the exploration sequence, however, we show that DSEE
achieves a regret arbitrarily close to the logarithmic order without
\emph{any} knowledge of the reward model. We further emphasize that
the sublinear regret for reward distributions with heavy tails is
achieved without any knowledge of the reward model (other than a
lower bound on the order of the highest finite moment).

\subsection{Extendability to Variations of MAB}
\label{sec:intro-variations}

Different from the classic policies proposed
in~\cite{Lai&Robbins85AAM,{AgrawalEtal95AAP},{Auer&etal02ML}}, the DSEE
approach has a
clearly defined tunable parameter---the cardinality of the
exploration sequence---which can be adjusted according to the
``hardness'' (in terms of learning) of the reward distributions and
observation models. It is thus more easily extendable to handle
variations of MAB, including decentralized MAB with multiple players and incomplete
reward observations under collisions, MAB with unknown Markov dynamics, and combinatorial MAB
with dependent arms that often arise in
network optimization problems such as the shortest path, the minimum spanning, and the dominating set problems
under unknown random weights.

Consider first a decentralized MAB problem in which multiple
distributed players learn from their local observations and make decisions independently. While other players' observations
and actions are unobservable, players' actions affect each other: conflicts occur when multiple players choose the same
arm at the same time and conflicting players can only share the reward offered by the arm, not necessarily with conservation.
Such an event is referred to as a collision and is unobservable to the players. In other words,
a player does not know whether it is
involved in a collision, or equivalently, whether the received
reward reflects the true state of the arm. Collisions thus not only
result in immediate reward loss, but also corrupt the observations
that a player relies on for learning the arm rank.
Such decentralized learning problems arise in communication networks where multiple distributed users share the access to a common set
of channels, each with unknown communication quality. If multiple users access the same channel at the same time, no one
transmits successfully or only one captures the channel through certain signaling schemes such as carrier sensing. Another
application is multi-agent systems in which $M$
agents search or collect targets in $N$ locations. When multiple
agents choose the same location, they share the reward in an unknown
way that may depend on which player comes first or the number of
colliding agents.

The deterministic separation of exploration and exploitation in
DSEE, however, can ensure that collisions are contained within the
exploitation sequence. Learning in the exploration sequence is thus
carried out using only reliable observations. In particular, we show that under the
DSEE approach, the system regret, defined as
the total reward loss with respect to the ideal scenario of known
reward models and centralized scheduling among players, grows at the
same orders as the regret in the single-player MAB under the same
conditions on the reward distributions. These results hinge on the
extendability of DSEE to targeting at arms with arbitrary ranks (not
necessarily the best arm) and the sufficiency in learning the arm
rank solely through the observations from the exploration sequence.

The DSEE approach can also be extended to MAB with unknown Markov reward models
and the so-called combinatorial MAB where there is a large number of arms dependent
through a smaller number of unknowns. Since these two extensions are
more involved and require separate investigations, they are not included in
this paper and can be found in~\cite{Liu&etal:13IT,Liu&Zhao:WiOpt12}.

\subsection{Related Work}\label{sec:intro_rw}

There have been a number of recent studies on extending the classic policies
of MAB to more general settings. In~\cite{Bubeck},
the UCB policy proposed by
Auer~\etal in~\cite{Auer&etal02ML} was extended to achieve logarithmic regret
order for heavy-tailed reward distributions when an upper bound on a finite moment is known.
The basic idea is to replace
the sample mean in the UCB index with a truncated mean estimator which allows
a mean concentration result similar to the Chernoff-Hoeffding bound. The computational
complexity and memory requirement of the resulting UCB policy, however, are much
higher since all past observations need to be stored and truncated differently at
each time $t$.
The result of
achieving the logarithmic regret order for heavy-tailed distributions under DSEE
in Sec.~\ref{sec:HeavyLog} is inspired by~\cite{Bubeck}. However, the focus of this paper is to present a general
approach to MAB, which not only provides a different policy for achieving logarithmic
regret order for both light-tailed and heavy-tailed distributions, but also offers
solutions to various MAB variations as discussed in~Sec.~\ref{sec:intro-variations}.
DSEE also offers the option of a sublinear regret order for heavy-tailed distributions
with a constant memory requirement, sublinear complexity, and no requirement on any
knowledge of the reward distributions (see Sec.~\ref{sec:HeavySub}).

In the context
of decentralized MAB with multiple players, the problem was
formulated in~\cite{Liu&Zhao:10TSP} with a simpler collision model:
regardless of the occurrence of collisions, each player always
observes the actual reward offered by the selected arm. In this
case, collisions affect only the immediate reward but not the
learning ability. It was shown that the optimal system regret has
the same logarithmic order as in the classic MAB with a single
player, and a Time-Division Fair sharing (TDFS) framework for
constructing order-optimal decentralized policies was proposed.
Under the same complete observation model,
decentralized MAB was also addressed in \cite{Anima,Gaiyi}, where
the single-player policy UCB1 was extended to the multi-player
setting under a Bernoulli reward model.
In~\cite{Tekin:3}, Tekin and Liu addressed decentralized learning under  general interference functions and light-tailed reward models.
In~\cite{Jain:12CDC,Jain:arx}, Kalathil \etal considered a more challenging case where arm ranks may be different across players
and addressed both i.i.d. and Markov reward models.
They proposed a decentralized policy that achieves near-$O(\log^2 T)$  regret for distributions with bounded support.
Different from this paper, all the above referenced work assumes complete reward observation under collisions
and focuses on specific light-tailed distributions.

\section{The Classic MAB}\label{sec:formulation}
Consider an $N$-arm bandit and a single player. At each time $t$,
the player chooses one arm to play. Playing arm $n$ yields i.i.d.
random reward $X_n(t)$ drawn from an unknown distribution $f_n(s)$.
Let $\Fc=(f_1(s),\cdots,f_N(s))$ denote the set of the unknown
distributions. We assume that the reward mean
$\theta_n\defeq\mathbb{E}[X_n(t)]$ exists for all $1\le n\le N$.

An arm selection policy $\pi$ is a function that maps from the
player's observation and decision history to the arm to play. Let
$\sigma$ be a permutation of $\{1,\cdots,N\}$ such that
$\theta_{\sigma(1)}\ge\theta_{\sigma(2)}\ge\cdots\ge\theta_{\sigma(N)}$.
The system performance under policy $\pi$ is measured by the regret
$R_T^{\pi}(\Fc)$ defined as
\[R_T^{\pi}(\Fc)\defeq T\theta_{\sigma(1)}-\mathbb{E}_{\pi}[\Sigma_{t=1}^TX_{\pi}(t)],\]
where $X_{\pi}(t)$ is the random reward obtained at time $t$ under
policy $\pi$, and $\mathbb{E}_{\pi}[\cdot]$ denotes the expectation
with respect to policy $\pi$. The objective is to minimize the rate
at which $R_T^{\pi}(\Fc)$ grows with $T$ under any distribution set
$\Fc$ by choosing an optimal policy $\pi^*$. We say that a policy is
order-optimal if it achieves a regret growing at the same order of
the optimal one. We point out that any policy with a sublinear
regret order achieves the maximum average
reward~$\theta_{\sigma(1)}$.

\section{The DSEE Approach}\label{sec:policy}

In this section, we present the DSEE approach and analyze its
performance for both light-tailed and heavy-tailed reward
distributions.

\subsection{The General Structure}\label{subsec:DSEEgenstruc}
Time is divided into two interleaving sequences: an exploration
sequence and an exploitation sequence. In the exploration sequence,
the player plays all arms in a round-robin fashion. In the
exploitation sequence, the player plays the arm with the largest
sample mean (or a properly chosen mean estimator) calculated from past reward observations. It is also
possible to use only the observations obtained in the exploration
sequence in computing the sample mean. This leads to the same
regret order with a significantly lower complexity since the sample
mean of each arm only needs to be updated at the same sublinear rate as the
exploration sequence. A detailed implementation of DSEE is given in
Fig~\ref{fig:policy_frame}.

\begin{figure}[htbp]
\begin{center}
\noindent\fbox{
\parbox{4.8in}
{ \centerline{\underline{{\bf The DSEE Approach}}} {
\begin{itemize}
\item Notations and Inputs: Let $\Ac(t)$ denote the set of time indices that belong to
the exploration sequence up to (and including) time $t$. Let
$|\Ac(t)|$ denote the cardinality of $\Ac(t)$. Let
$\overline{\theta}_{n}(t)$ denote the sample mean of arm $n$
computed from the reward observations at times in $\Ac(t-1)$. For
two positive integers $k$ and $l$, define $k\oslash l\defeq
((k-1)~\mbox{mod}~l)+1$, which is an integer taking values from
$1,2,\cdots,l$.\\[-1.5em]
\item At time $t$,
\begin{enumerate}
\item[1.] if $t\in\Ac(t)$, play arm $n=|\Ac(t)|\oslash N$;
\item[2.] if $t\notin\Ac(t)$, play arm $n^*=\arg\max\{\overline{\theta}_{n}(t),~1\le n\le N\}$.
\end{enumerate}
\end{itemize}} }} \caption{The DSEE approach for the classic MAB.}\label{fig:policy_frame}
\end{center}
\end{figure}

In DSEE, the tradeoff between exploration and exploitation is
balanced by choosing the cardinality of the exploration sequence. To
minimize the regret growth rate, the cardinality of the exploration
sequence should be set to the minimum that ensures a reward loss in
the exploitation sequence having an order no larger than the
cardinality of the exploration sequence. The detailed regret analysis
is given in the next subsection.

\subsection{Under Light-Tailed Reward Distributions}\label{subsec:opt}
In this section, we construct an exploration sequence in DSEE to
achieve the optimal logarithmic regret order for all light-tailed
reward distributions.

We
recall the definition of light-tailed distributions below.

\begin{definition}\label{def:mgf}
A random variable $X$ is light-tailed if its moment-generating
function exists, \ie there exists a $u_0>0$ such that for
all~$u\le|u_0|$, {\[M(u)\defeq\mathbb{E}[\exp(uX)]<\infty.\]}
Otherwise $X$ is heavy-tailed.
\end{definition}

For a zero-mean light-tailed random variable $X$, we
have~\cite{{CharekaEtal06JM}},
\begin{eqnarray}\label{eqn:mgfbound}
M(u)\le\exp(\zeta
u^2/2),~~~~\forall~u\le|u_0|,~\zeta\ge\sup\{M^{(2)}(u),~-u_0\le u\le
u_0\},
\end{eqnarray}
where $M^{(2)}(\cdot)$ denotes the second derivative of
$M(\cdot)$ and $u_0$ the parameter specified in Definition~\ref{def:mgf}. We
observe that the upper bound in~\eqref{eqn:mgfbound} is the
moment-generating function of a zero-mean Gaussian random variable
with variance $\zeta$. Thus, light-tailed distributions
are also called \emph{locally} sub-Gaussian
distributions. If the moment-generating function exists for all $u$,
the corresponding distributions are referred to as sub-Gaussian.
From~\eqref{eqn:mgfbound}, we have the following extended
Chernoff-Hoeffding bound on the deviation of the sample mean.

\begin{lemma}\label{lemma:chernoff}{\em (Chernoff-Hoeffding
Bound~\cite{Agrawal95SIAMJCO})} Let $\{X(t)\}_{t=1}^{\infty}$ be
i.i.d. random variables drawn from a light-tailed distribution. Let
$\overline{X_s}=(\Sigma_{t=1}^s X(t))/s$ and
$\theta=\mathbb{E}[X(1)]$. We have, for all $\delta\in [0,\zeta
u_0], a\in (0, \frac{1}{2\zeta}]$,
\begin{eqnarray}\label{eqn:chernoff}
\Pr(|\overline{X_s}-\theta|\ge \delta)\le 2\exp(-a\delta^2s).
\end{eqnarray}
\end{lemma}
\vspace{0.5em}

Proven in~\cite{Agrawal95SIAMJCO}, Lemma~\ref{lemma:chernoff}
extends the original Chernoff-Hoeffding bound given
in~\cite{Hoeffding63JASA} that considers only random variables with
a bounded support. Based on Lemma~\ref{lemma:chernoff}, we show in
the following theorem that DSEE achieves the optimal logarithmic
regret order for all light-tailed reward distributions.

\begin{theorem}\label{thm:opt}
Construct an exploration sequence as follows. Let $a,\zeta,u_0$ be
the constants such that~\eqref{eqn:chernoff} holds. Define
$\Delta_{n}\defeq \theta_{\sigma(1)}-\theta_{\sigma(n)}$ for $n=2,\ldots,N$.
Choose a constant
$c\in(0,\Delta_2)$, a constant $\delta=\min\{c/2,\zeta u_0\}$, and a constant
$w> \frac{1}{a\delta^2}$. For each $t>1$, if
$|\Ac(t-1)|<N\lceil w\log t\rceil$, then include $t$ in $\Ac(t)$.
Under this exploration sequence, the resulting DSEE policy $\pi^*$
has regret, $\forall T$,
\begin{eqnarray}\label{eqn:optregret}
R_T^{\pi^*}(\Fc)\le \Sigma_{n=2}^{N}\lceil w\log T\rceil \Delta_n+2N\Delta_{N}(1+\frac{1}{a\delta^2w-1}).
\end{eqnarray}
\end{theorem}

\begin{proof}
Without loss of generality, we assume that $\{\theta_n\}_{n=1}^N$
are distinct. Let $R_{T,O}^{\pi^*}(\Fc)$ and $R_{T,I}^{\pi^*}(\Fc)$ denote, respectively,
regret incurred during the exploration and the exploitation sequences.
From the construction of the exploration sequence, it is easy to see that
{\begin{eqnarray}
R_{T,O}^{\pi^*}(\Fc)\le\Sigma_{n=2}^{N}\lceil w\log T\rceil \Delta_n.
\label{eq:RTO1}
\end{eqnarray}}

During the exploitation sequence, a reward loss happens if the
player incorrectly identifies the best arm. We thus have
\begin{eqnarray} \nn
R_{T,I}^{\pi^*}(\Fc)&\leq&
\mathbb{E}[\Sigma_{t\notin\Ac(T),t\leq T}\mathbb{I}(\pi^*(t)\neq
\sigma(1))]\Delta_{N}\\
&=&\Sigma_{t\notin\Ac(T),t\leq T}\Pr(\pi^*(t)\neq \sigma(1))\Delta_{N}.
\end{eqnarray}
For $t\notin\Ac(T)$,
define the following event
{\begin{eqnarray}\label{Event}
\Ec(t)\defeq\{|\overline{\theta}_{n}(t)-\theta_n|\le\delta,~\forall~1\le
n\le N\}.
\end{eqnarray}}
From the choice of $\delta$, it is easy to see that under $\Ec(t)$, the arm ranks are correctly identified. We thus have
\begin{eqnarray}\nn
R_{T,I}^{\pi^*}(\Fc)&\le& \Sigma_{t\notin\Ac(T),t\leq T}\Pr(\overline{\Ec(t)})\Delta_{N} \\ \nn
&=&\Sigma_{t\notin\Ac(T),t\leq T}\Pr(\exists~1\le n\le N~\mbox{s.t.}~|\overline{\theta}_{n}(t)-\theta_{n}|> \delta)\Delta_{N}\\
&\le&\Sigma_{t\notin\Ac(T),t\leq T} \Sigma_{n=1}^{N} \Pr(|\overline{\theta}_{n}(t)-\theta_{n}|> \delta)\Delta_{N},
\label{eq:RTI1}
\end{eqnarray}
where \eqref{eq:RTI1} results from the union bound.
Let $\tau_{n}(t)$ denote the number of times that arm $n$ has been played during the exploration sequence up to time $t$.
Applying Leamma ~\ref{lemma:chernoff} to \eqref{eq:RTI1}, we have
\begin{eqnarray}
R_{T,I}^{\pi^*}(\Fc)&\le& 2\Delta_N \Sigma_{t\notin\Ac(T),t\leq T}\Sigma_{n=1}^{N}\exp(-a\delta^2\tau_n(t)) \nn \\
&\le& 2\Delta_N \Sigma_{t\notin\Ac(T),t\leq T}\Sigma_{n=1}^{N}\exp(-a\delta^2w\log t) \label{eqn:LogLT}\\
&=& 2\Delta_N \Sigma_{t\notin\Ac(T),t\leq T}\Sigma_{n=1}^{N}t^{-a\delta^2w}\nn \\
&\le& 2N\Delta_N \Sigma_{t=1}^{\infty}t^{-a\delta^2w} \nn \\
&\le& 2N\Delta_{N}(1+\frac{1}{a\delta^2w-1}), \label{eqn:LogLT2}
\end{eqnarray}
where \eqref{eqn:LogLT} comes from $\tau_{n}(t) \ge w \log t$ and \eqref{eqn:LogLT2} from $a\delta^2w>1$.

Combining \eqref{eq:RTO1} and \eqref{eqn:LogLT2}, we arrive at the theorem.
\end{proof}

\vspace{.5em}

The choice of the exploration sequence given in Theorem~\ref{thm:opt} is not unique. In particular,
when the horizon length $T$ is given, we can choose a single block of exploration
followed by a single block of exploitation. In the case of infinite horizon,
we can follow the standard technique of partitioning the time horizon
into epochs with geometrically growing lengths and applying the finite-$T$
scheme to each epoch.

We point out that the logarithmic regret order requires certain knowledge about the
differentiability of the best arm. Specifically, we need a lower
bound (parameter $c$ defined in Theorem~\ref{thm:opt}) on the
difference in the reward mean of the best and the second best arms.
We also need to know the bounds on parameters $\zeta$ and $u_0$ such
that the Chernoff-Hoeffding bound~\eqref{eqn:chernoff} holds. These
bounds are required in defining $w$ that specifies the minimum
leading constant of the logarithmic cardinality of the exploration
sequence necessary for identifying the best arm. However, we show
that when no knowledge on the reward models is available, we can increase the
cardinality of the exploration sequence of $\pi^*$ by an arbitrarily
small amount to achieve a regret arbitrarily close to the
logarithmic order.

\begin{theorem}\label{thm:nearlog}
Let $f(t)$ be any positive increasing sequence with
$f(t)\rightarrow\infty$ as $t\rightarrow\infty$. Construct an exploration sequence as follows.
For each $t>1$, include $t$ in
$\Ac(t)$ if $|\Ac(t-1)|<N\lceil f(t)\log t\rceil$. The resulting DSEE policy $\pi^*$ has regret
\begin{eqnarray}\nn
R_{T}^{\pi^*}(\Fc)=O(f(T)\log T).
\end{eqnarray}
\end{theorem}
\begin{proof}
Recall constants $a$ and $\delta$ defined in Theorem~\ref{thm:opt}.
Note that since $f(t)\rightarrow\infty$ as $t\rightarrow\infty$, there exists a $t_0$ such that for any $t> t_0$, $a\delta^2f(t)\ge b$ for some $b>1$.
Similar to the proof of Theorem~\ref{thm:opt}, we have, following \eqref{eq:RTI1},
\begin{eqnarray}\nn
R_{T,I}^{\pi^*}(\Fc)&\le&2N\Delta_N \Sigma_{t\notin\Ac(T),t\leq T}\exp(-a\delta^2f(t)\log t)\\\nn
&\le& \Sigma_{t=1}^{t_0}\exp(-a\delta^2f(t)\log t)+\Sigma_{t=t_0+1}^{\infty}t^{-b}\\
&\le& t_0+\frac{1}{b-1}t_0^{1-b}.
\label{eq:RTI2}
\end{eqnarray}
It is easy to see that
{\begin{eqnarray}
R_{T,O}^{\pi^*}(\Fc)\le\Sigma_{n=2}^{N}\lceil f(T)\log T\rceil \Delta_n.
\label{eq:RTO2}
\end{eqnarray}}
Combining \eqref{eq:RTI2} and \eqref{eq:RTO2}, we have
\begin{equation}
R_{T}^{\pi^*}(\Fc) \le \sum_{n=2}^{N}\lceil f(T)\log T\rceil \Delta_n \, +\, t_0+\frac{1}{b-1}t_0^{1-b}.
\label{eq:nearLogConst}
\end{equation}
\end{proof}

From the proof of Theorem~\ref{thm:nearlog}, we observe
a tradeoff between the regret order and the finite-time performance.
While one can arbitrarily approach the logarithmic regret order by reducing
the diverging rate of $f(t)$, the price is a larger additive constant as shown in~\eqref{eq:nearLogConst}.

\subsection{Under Heavy-Tailed Reward Distributions}\label{subsec:gen}

In this subsection, we consider the regret performance of DSEE under heavy-tailed reward distributions.

\subsubsection{Sublinear Regret with Sublinear Complexity and No Prior Knowledge}
\label{sec:HeavySub}

For heavy-tailed reward distributions, the Chernoff-Hoeffding bound
does not hold in general. A weaker bound on the deviation of the
sample mean from the true mean is established in the lemma below.
\begin{lemma}\label{lemma:twop}
Let $\{X(t)\}_{t=1}^{\infty}$ be i.i.d. random variables drawn from
a distribution with finite $p$th moment ($p>1$). Let $\overline{X}_t=\frac{1}{t}\Sigma_{k=1}^t
X(k)$ and $\theta=\mathbb{E}[X(1)]$. We have, for all
$\delta>0$,
\begin{eqnarray}
\Pr(|\overline{X}_t-\theta|\ge\delta)\le \left\{
\begin{array}{cc} (3\sqrt2)^pp^{p/2} \frac{\mathbb{E}[|X(1)-\theta|^p]}{\delta^p}t^{1-p}~~~& \text{if}~~~ p\le2\\
(3\sqrt2)^pp^{p/2} \frac{\mathbb{E}[|X(1)-\theta|^p]}{\delta^p}t^{-p/2}~~~& \text{if}~~~ p>2
\end{array}\right.
\end{eqnarray}

\end{lemma}
\begin{proof}
By Chebyshev's inequality we have,
{\begin{eqnarray}\label{eqn:Zygmund}\nn
\Pr(|\overline{X}_t-\theta|\ge\delta)&\le& \frac{\mathbb{E}[|\overline{X}_t-\theta|^p]}{\delta^p}\\ \nn
&=&\frac{\mathbb{E}[|\Sigma_{k=1}^t(X(k)-\theta)|^p]}{t^p\delta^p}\\
&\le& B_p\frac{\mathbb{E}[(\Sigma_{k=1}^t(X(k)-\theta)^2)^{p/2}]}{t^p\delta^p},
\end{eqnarray}}
where~\eqref{eqn:Zygmund} holds by the Marcinkiewicz-Zygmund inequality for some $B_p$ depending only on $p$.
The best constant in the Marcinkiewicz-Zygmund inequality was shown in~\cite{Zygmund} to be $B_p\le(3\sqrt2)^pp^{p/2}$.

Next, we prove Lemma~\ref{lemma:twop} by considering the two cases of $p$.
\begin{itemize}
\item $p\le2$:
Considering the inequality $(\Sigma_{k=1}^t a_k)^\alpha\le\Sigma_{k=1}^t a_k^\alpha$ for $a_k\ge 0$ and $\alpha\le1$
(which can be easily shown using induction), we have, from \eqref{eqn:Zygmund},
{\begin{eqnarray}\nn
\Pr(|\overline{X}_t-\theta|\ge\delta)&\le&B_p\frac{\mathbb{E}[\Sigma_{k=1}^t|X(k)-\theta|^p]}{t^p\delta^p}\\
&=&B_p \frac{\mathbb{E}[|X(1)-\theta|^p]}{\delta^p}t^{1-p}.
\end{eqnarray}}
\item $p>2$:
Using Jensen's inequality, we have, from \eqref{eqn:Zygmund},
{\begin{eqnarray}\nn
\Pr(|\overline{X}_t-\theta|\ge\delta)&\le&B_p\frac{\mathbb{E}[t^{p/2-1}\Sigma_{k=1}^t|X(k)-\theta|^p]}{t^p\delta^p}\\
&=&B_p \frac{\mathbb{E}[|X(1)-\theta|^p]}{\delta^p}t^{-p/2}.
\end{eqnarray}}
\end{itemize}

\end{proof}
Based on Lemma~\ref{lemma:twop}, we have the following results
on the regret performance of DSEE under heavy-tailed reward distributions.

\begin{theorem}\label{thm:general}
Assume that the reward distributions have finite $p$th order moment ($p>1$).
Construct an exploration sequence as follows. Choose a constant
$v>0$. For each $t>1$, include $t$ in
$\Ac(t)$ if $|\Ac(t-1)|<vt^{1/p}$ for $1<p\le 2$ or $|\Ac(t-1)|<vt^{\frac{1}{1+p/2}}$ for $p>2$.
Under this exploration sequence, the resulting DSEE policy
$\pi^p$ has regret
\begin{eqnarray}
R^{\pi^p}_T(\Fc)\le\left\{
\begin{array}{ll}
O(T^{1/p})& \text{if}~~ 1<p\le2\\
O(T^{\frac{1}{1+p/2}})& \text{if}~~ p>2
\end{array}\right.
\end{eqnarray}
\end{theorem}
An upper bound on the regret for each $T$ is given in~\ref{eq:HeavySub} in the proof.

\begin{proof}
We prove the theorem for the case of $p\ge2$, the other case can be shown similarly. Following a similar line of arguments as in the proof of Theorem~\ref{thm:opt},
we can show, by applying Lemma~\ref{lemma:twop} to \eqref{eq:RTI1},
\begin{eqnarray}\label{eqn:3}\nn
R^{\pi^p}_{T,I}(\Fc)&\le&\Delta_N (3\sqrt2)^pp^{p/2} \frac{\mathbb{E}[|X(1)-\theta|^p]}{\delta^p}v^{-p/2}\Sigma_{t=1}^Tt^{\frac{-p/2}{1+p/2}}\\
&\le&\Delta_N (3\sqrt2)^pp^{p/2} \frac{\mathbb{E}[|X(1)-\theta|^p]}{\delta^p}v^{-p/2}[(1+p/2)(T^{\frac{1}{1+p/2}}-1)+1]\\\nn
\end{eqnarray}
Considering the cardinality of the exploration sequence, we have, $\forall T$,
\begin{eqnarray}
R^{\pi^p}_T(\Fc)\le\left\{
\begin{array}{ll}
\Delta_N(3\sqrt2)^pp^{p/2} \frac{\mathbb{E}[|X(1)-\theta|^p]}{(\Delta_2/2)^p}v^{-p/2}[p(T^{\frac{1}{p}}-1)+1]+\lceil vT^{\frac{1}{p}}\rceil & \text{if } p\le2\\
\Delta_N(3\sqrt2)^pp^{p/2} \frac{\mathbb{E}[|X(1)-\theta|^p]}{(\Delta_2/2)^p}v^{-p/2}[(1+p/2)(T^{\frac{1}{1+p/2}}-1)+1]+\lceil vT^{\frac{1}{1+p/2}}\rceil & \text{if } p>2
\end{array}\right.
\label{eq:HeavySub}
\end{eqnarray}

\vspace{0.2em}

The regret order given in Theorem~\ref{thm:general} is thus readily seen.
\end{proof}

\vspace{1em}

\subsubsection{Logarithmic Regret Using Truncated Sample Mean}
\label{sec:HeavyLog}

Inspired by the work by Bubeck, Cesa-Bianchi, and Lugosi~\cite{Bubeck},
we show in this subsection that using the truncated sample mean, DSEE can offer
logarithmic regret order for heavy-tailed reward distributions with a carefully
chosen cardinality of the exploration sequence. Similar to the UCB variation developed
in~\cite{Bubeck}, this logarithmic regret order is achieved at the price of
prior information on the reward distributions and higher computational and memory requirement.
The computational and memory requirement, however, is significantly lower than that of the UCB variation
in~\cite{Bubeck}, since the DSEE approach only needs to store samples from and compute the truncated
sample mean at the exploration times with $O(\log T)$ order rather than each time instant.

The main idea is based on the following result on the truncated sample mean given in~\cite{Bubeck}.

\begin{lemma}(\cite{Bubeck}:) \label{lemma:deviation} Let
$\{X(t)\}_{t=1}^{\infty}$ be i.i.d. random variables satisfying
$\mathbb{E}[|X(1)|^p]\leq u$ for some constants $u>0$ and $p \in (1,2]$.
Let $\theta=\mathbb{E}[X(1)]$. Consider the truncated empirical mean $\widehat{\theta}(s,\epsilon)$ defined as
\begin{eqnarray}\label{eqn:TruncatedMean}
\widehat{\theta}(s,\epsilon)=\frac{1}{s}\sum_{t=1}^{s}X(t)\mathbbm{1}\{ |X(t)|\leq(\frac{ut}{\log(\epsilon^{-1})})^{1/{p}} \}.
\end{eqnarray}
Then for any $\epsilon \in (0,\frac{1}{2}]$,
{\begin{eqnarray} \label{eqn:LemIneq}
\Pr(|\widehat{\theta}(s,\epsilon)-\theta|>4u^{1/{p}}(\frac{\log(\epsilon^{-1})}{s})^{\frac{{p}-1}{{p}}})\leq 2\epsilon.
\end{eqnarray}}

\end{lemma}

\vspace{0.5em}

Based on Lemma~\ref{lemma:deviation}, we have the following result on the regret of DSEE.

{\begin{theorem}\label{thmHTReg}
Assume that the reward of each arm satisfies $\mathbb{E}[|X_n(1)|^{p}]\leq u$ for some constants $u>0$ and $p \in (1,2]$.
Let $a=4^{\frac{{p}}{1-{p}}}u^{\frac{1}{1-p}}$. Define
$\Delta_{n}\defeq \theta_{\sigma(1)}-\theta_{\sigma(n)}$ for $n=2,\ldots,N$.
Construct an exploration sequence as follows. Choose a constant $\delta\in(0,\, \Delta_2/2)$
and a constant $w>1/{a\delta^{{p}/({p}-1)}}$. For each $t>1$, if
$|\Ac(t-1)|<N\lceil w\log t\rceil$, then include $t$ in $\Ac(t)$.
At an exploitation time $t$, play the arm with the largest truncated sample mean given by
\begin{eqnarray}
\widehat{\theta}_n(\tau_n(t),\epsilon_n(t))=\frac{1}{\tau_n(t)}\sum_{k=1}^{\tau_n(t)}X_{n,k}
\mathbbm{1}\{ |X_{n,k}|\leq(\frac{uk}{\log(\epsilon_n(t)^{-1})})^{1/{p}} \},
\end{eqnarray}
where $X_{n,k}$ denotes the $k$th observation of arm $n$ during the exploration sequence, $\tau_n(t)$ the total
number of such observations, and $\epsilon_n(t)$ in the truncator for each arm at each time $t$ is given by
\begin{eqnarray}\label{eqn:eps}
\epsilon_n(t)=\exp(-a\delta^{\frac{{p}}{{p}-1}}\tau_n(t)).
\end{eqnarray}
The resulting DSEE policy $\pi^*$
has regret
\begin{eqnarray}\label{eqn:optregret2}
R_T^{\pi^*}(\Fc)\le \sum_{n=2}^{N}\lceil w\log T\rceil \Delta_n +2N\Delta_{N}(1+\frac{1}{a\delta^{p/(p-1)}w-1}).
\end{eqnarray}
\end{theorem}}

\begin{proof}
Following the same line of arguments as in the proof of Theorem~~\ref{thm:opt}, we have, following \eqref{eq:RTI1}
\begin{equation}\label{eqn:bounddev}
R_{T,I}^{\pi^*}(\Fc)\le \sum_{t\notin\Ac(T),t\leq T} \sum_{n=1}^{N}
\Pr(|\widehat{\theta}_{n}(\tau_{n}(t),\epsilon_{n}(t))-\theta_{n}|> \delta)\Delta_{N}.
\end{equation}
Based on Leamma~\ref{lemma:deviation}, we have, by
substituting $\epsilon_n(t)$ given in~\eqref{eqn:eps} into~\eqref{eqn:LemIneq},
\begin{equation}
\Pr(|\widehat{\theta}(\tau_n(t),\epsilon_n(t))-\theta_n|>\delta)\leq 2\exp(-a\delta^{\frac{p}{p-1}}\tau_n(t)).
\end{equation}
Substituting the above equation into \eqref{eqn:bounddev}, we have
{\begin{eqnarray}\label{eqn:LogHT}\nn
R_{T,I}^{\pi^*}(\Fc)&\le& 2\Delta_N \Sigma_{t\notin\Ac(T),t\leq T}\Sigma_{n=1}^{N}\exp(-a\delta^{\frac{p}{p-1}}\tau_n(t))\\\nn
&\le& 2\Delta_N \Sigma_{t\notin\Ac(T),t\leq T}\Sigma_{n=1}^{N}\exp(-a\delta^{\frac{p}{p-1}}w\log t)\\ \nn
&=& 2\Delta_N \Sigma_{t\notin\Ac(T),t\leq T}\Sigma_{n=1}^{N}t^{-a\delta^{p/(p-1)}w}\\\nn
&\le& 2N\Delta_N \Sigma_{t=1}^{\infty}t^{-a\delta^{p/(p-1)}w}\\\label{eqn:LogHT2}
&\le& 2N\Delta_{N}(1+\frac{1}{a\delta^{p/(p-1)}w-1})
\end{eqnarray}}
We then arrive at the theorem, considering $R_{T,O}^{\pi^*}(\Fc)\le\Sigma_{n=2}^{N}\lceil w\log T\rceil \Delta_n$.
\end{proof}

We point out that to achieve the logarithmic regret order under heavy-tailed distributions,
an upper bound on $\mathbb{E}[|X_n(1)|^{p}]$ for a certain $p$ needs to be known. The range constraint
of $p \in (1,2]$ in Theorem~\ref{thmHTReg} can be easily addressed:
if we know $\mathbb{E}[|X_n(1)|^{p}]\le u$ for a certain $p>2$, then $\mathbb{E}[|X|^2]\leq u+1$.
Similar to Theorem~\ref{thm:nearlog}, we can show that when no knowledge on the reward models is available, we can increase the
cardinality of the exploration sequence by an arbitrarily
small amount (any diverging sequence $f(t)$) to achieve a regret arbitrarily close to the
logarithmic order. One necessary change to the policy is that the constant $\delta$ in Theorem~\ref{thmHTReg}
used in \eqref{eqn:eps} for calculating the truncated sample mean should be replaced by $f(t)^\gamma$ for
some $\gamma\in (\frac{1-p}{p}, 0)$.

\section{Variations of MAB}\label{sec:DMAB_IM}
In this section, we extend the DSEE approach to several MAB variations
including MAB with various
objectives, decentralized MAB with multiple players and incomplete
reward observations under collisions, MAB with unknown Markov dynamics, and combinatorial MAB with
dependent arms.

\subsection{MAB under Various Objectives}
Consider a generalized MAB problem in which the desired arm is the
$m$th best arm for an arbitrary $m$. Such objectives may arise when
there are multiple players (see the next subsection) or other
constraints/costs in arm selection. The classic policies
in~\cite{Lai&Robbins85AAM,{AgrawalEtal95AAP},{Auer&etal02ML}} cannot
be directly extended to handle this new objective. For example, for
the UCB policy proposed by Auer~\etal in~\cite{Auer&etal02ML},
simply choosing the arm with the $m$th $(1<m\le N)$ largest index
cannot guarantee an optimal solution. This can be seen from the
index form given in~\eqref{eqn:AuerIndex}: when the index of the
desired arm is too large to be selected, its index tends to become
even larger due to the second term of the index. The rectification
proposed in~\cite{Gai&Krish11TR} is to combine the upper confidence
bound with a symmetric lower confidence bound. Specifically, the arm
selection is completed in two steps at each time: the upper
confidence bound is first used to filter out arms with a lower rank,
the lower confidence bound is then used to filter out arms with a
higher rank. It was shown in~\cite{Gai&Krish11TR} that under the
extended UCB, the expected time that the player does not play the
targeted arm has a logarithmic order.

The DSEE approach, however, can be directly extended to handle this
general objective. Under DSEE, all arms, regardless of their ranks,
are sufficiently explored by carefully choosing the cardinality of
the exploration sequence. As a consequence, this general objective
can be achieved by simply choosing the arm with the $m$th largest
sample mean in the exploitation sequence. Specifically, assume that
a cost $C_j>0~(j\neq m,1\le j\le N)$ is incurred when the player
plays the $j$th best arm. Define the regret $R_T^{\pi}(\Fc,m)$ as
the expected total costs over time $T$ under policy $\pi$.

\begin{theorem}\label{thm:mbest}
By choosing the parameter $c$ in Theorem~\ref{thm:opt} to satisfy
%$0<c<\min\{\theta_{\sigma(m-1)}-\theta_{\sigma(m)},\theta_{\sigma(m)}-\theta_{\sigma(m+1)}\}$
$0<c<\min\{\Delta_m-\Delta_{m-1},\Delta_{m+1}-\Delta_m\}$
or a
parameter $\delta$ in theorem~\ref{thm:general} and~\ref{thmHTReg} to satisfy %$0<\delta<\frac{1}{2}\min\{\theta_{\sigma(m-1)}-\theta_{\sigma(m)},\theta_{\sigma(m)}-\theta_{\sigma(m+1)}\}$
$0<\delta<\frac{1}{2}\min\{\Delta_m-\Delta_{m-1},\Delta_{m+1}-\Delta_m\}$
and letting the player select the arm with the $m$-th largest sample
mean (or truncated sample mean in case of ~\ref{thmHTReg}) in the exploitation sequence,
Theorems~\ref{thm:opt}-\ref{thmHTReg} hold for $R_T^{\pi}(\Fc,m)$.
\end{theorem}
\begin{proof}
The proof is similar to those of previous theorems. The key
observation is that after playing all arms sufficient times during
the exploration sequence, the probability that the sample mean of
each arm deviates from its true mean by an amount larger than the
non-overlapping neighbor is small enough to
ensure a properly bounded regret incurred in the exploitation
sequence.
\end{proof}

\vspace{.5em}

We now consider an alternative scenario that the player targets at a
set of best arms, say the $M$ best arms. We assume that a cost is
incurred whenever the player plays an arm not in the set. Similarly,
we define the regret $R_T^{\pi}(\Fc,M)$ as the expected total costs
over time $T$ under policy $\pi$.

\begin{theorem}\label{thm:Mbest}
By choosing the parameter $c$ in Theorem~\ref{thm:opt} to satisfy
%$0<c<\theta_{\sigma(M)}-\theta_{\sigma(M+1)}$
$0<c<\Delta_{M+1}-\Delta_M$
or a
parameter $\delta$ in theorem~\ref{thm:general} and~\ref{thmHTReg} to
satisfy
%$0<\delta<\frac{1}{2}(\theta_{\sigma(M)}-\theta_{\sigma(M+1)})$
$0<\delta<\frac{1}{2}(\Delta_{M+1}-\Delta_M)$
and letting the player
select one of the $M$ arms with the largest sample means (or truncated sample mean in case of ~\ref{thmHTReg}) in the
exploitation sequence, Theorem~\ref{thm:opt}-\ref{thmHTReg} hold for
$R_T^{\pi}(\Fc,M)$.
\end{theorem}
\begin{proof}
The proof is similar to those of previous theorems. Compared to
Theorem~\ref{thm:mbest}, the condition on $c$ for applying
Theorem~\ref{thm:opt} is more relaxed: we only need to know a lower
bound on the mean difference between the $M$-th best and the
$(M+1)$-th best arms. This is due to the fact that we only need to
distinguish the $M$ best arms from others instead of specifying
their rank.
\end{proof}

\vspace{.5em}

By selecting arms with different ranks of the sample mean in the
exploitation sequence, it is not difficult to see that
Theorem~\ref{thm:mbest} and Theorem~\ref{thm:Mbest} can be applied
to cases with time-varying objectives. In the next subsection, we
use these extensions of DSEE to solve a class of decentralized MAB
with incomplete reward observations.

\subsection{Decentralized MAB with Incomplete Reward Observations}

\subsubsection{Distributed Learning under Incomplete Observations}

Consider $M$ distributed players.
At each time $t$, each player chooses one arm to play. When multiple
players choose the same arm (say, arm $n$) to play at time $t$, a
player (say, player $m$) involved in this collision obtains a
potentially reduced reward $Y_{n,m}(t)$ with $\sum_{m=1}^M
Y_{n,m}(t)\le X_n(t)$. We focus on the case where the $M$ best arms have positive
reward mean and collisions cause reward loss.
The distribution of the partial reward $Y_{n,m}(t)$ under collisions can take any unknown form and has
any dependency on $n$, $m$ and $t$. Players make decisions solely
based on their partial reward observations $Y_{n,m}(t)$ without information exchange.
Consequently, a player does not know whether it is involved in a
collision, or equivalently, whether the received reward reflects the
true state $X_n(t)$ of the arm.

A local arm selection policy $\pi_m$ of player $m$ is a function
that maps from the player's observation and decision history to the
arm to play. A decentralized arm selection policy $\pi$ is thus
given by the concatenation of the local polices of all players:
{\[\pi_d\defeq[\pi_1,\cdots,\pi_M].\]}The system performance under
policy $\pi_d$ is measured by the system regret $R_T^{\pi_d}(\Fc)$
defined as the expected total reward loss up to time $T$ under
policy $\pi_d$ compared to the ideal scenario that players are
centralized and $\Fc$ is known to all players (thus the $M$ best
arms with highest means are played at each time). We have
\[R_T^{\pi_d}(\Fc)\defeq T\Sigma_{n=1}^M\theta_{\sigma(n)}-\mathbb{E}_{\pi}[\Sigma_{t=1}^TY_{\pi_d}(t)],\]
where $Y_{\pi_d}(t)$ is the total random reward obtained at time $t$
under decentralized policy $\pi_d$. Similar to the single-player
case, any policy with a sublinear order of regret would achieve the
maximum average reward given by the sum of the $M$ highest reward
means.

\vspace{.5em}

\subsubsection{Decentralized Policies under DSEE} In order to minimize the system regret, it is crucial
that each player extracts reliable information for learning the arm
rank. This requires that each player obtains and recognizes
sufficient observations that were received without collisions. As
shown in Sec.~\ref{sec:policy}, efficient learning can be achieved
in DSEE by solely utilizing the observations from the deterministic
exploration sequence. Based on this property, a decentralized arm
selection policy can be constructed as follows.
In the
exploration sequence, players play all arms in a round-robin fashion
with different offsets which can be predetermined based on, for
example, the players' IDs, to eliminate collisions. In the
exploitation sequence, each player plays the $M$ arms with the
largest sample mean calculated using only observations from the
exploration sequence under either a prioritized or a fair sharing
scheme. While collisions still occur in the exploitation sequences
due to the difference in the estimated arm rank across players
caused by the randomness of the sample means, their effect on the
total reward can be limited through a carefully designed cardinality
of the exploration sequence.
Note that under a
prioritized scheme, each player needs to learn the specific rank of
one or multiple of the $M$ best arms and Theorem~\ref{thm:mbest} can
be applied. While under a fair sharing scheme, a player only needs
to learn the set of the $M$ best arms (as addressed in
Theorem~\ref{thm:Mbest}) and use the common arm index for fair
sharing. An example based on a round-robin fair sharing scheme is
illustrated in Fig.~\ref{fig:DSEE_decen}. We point out that under a
fair sharing scheme, each player achieves the same average reward at
the same rate.

\begin{figure}[h]
\centerline{
\begin{psfrags}
\psfrag{dot}[c]{$\cdots$} \psfrag{p1}[c]{\footnotesize~~~~~~~~~
Player 1}\psfrag{p2}[c]{\footnotesize~~~~~~~~~ Player 2}
\psfrag{T}[c]{$t$}\psfrag{o}[l]{Exploration}\psfrag{i}[l]{Exploitation}
\psfrag{e}[l]{\footnotesize Estimated best
arms}\psfrag{e1}[l]{\footnotesize
$(1,2)$}\psfrag{e2}[l]{\footnotesize
$(1,2)$}\psfrag{e3}[l]{\footnotesize $(2,3)$}
\psfrag{e4}[l]{\footnotesize $(1,2)$} \scalefig{1}\epsfbox{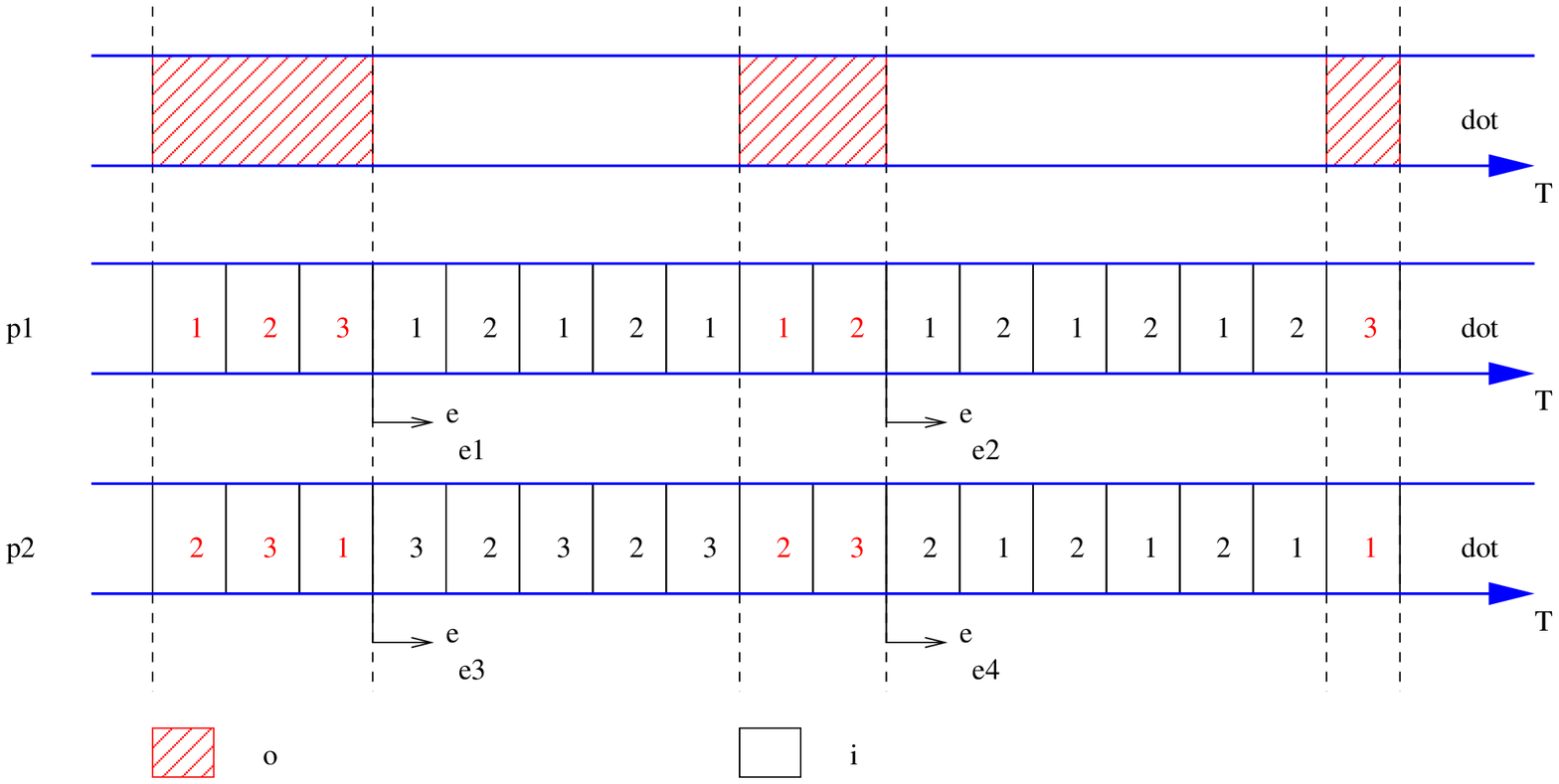}
\end{psfrags}}
\caption{An example of decentralized policies based on
DSEE~($M=2$,~$N=3$, the index of the selected arm at each time is
given).} \label{fig:DSEE_decen}
\end{figure}

\begin{theorem}
Under a decentralized policy based on DSEE, Theorem~\ref{thm:opt}-\ref{thmHTReg} hold for
$R_T^{\pi_d}(\Fc)$.
\end{theorem}
\begin{proof}
It is not difficult to see that the regret in the decentralized
policy is completely determined by the learning efficiency of the
$M$ best arms at each player.
The key is to notice that during the exploitation sequence, collisions can only happen if at least one
player incorrectly identifies the $M$ best arms. As a consequence, to analyze the regret in the exploitation sequence,
we only need to consider such events. The proof is
thus similar to
those of previous theorems.

%We give details regrading theorem~\ref{thm:opt}. One can easily get the idea
%to extend theorems~\ref{thm:nearlog}-\ref{thmHTReg}.
%It is easy to see that:
%{\begin{eqnarray}
%R_{T,O}^{\pi_d}(\Fc)\le M\Sigma_{n=1}^{N}\lceil w\log T\rceil \Delta_n.
%\end{eqnarray}}
%During exploitation sequence a mistake including collisions can be happened if one of the players
%fails to identify M best arms correctly. So taking a look at proof of theorem~\ref{thm:opt} we have
%
%{\begin{eqnarray}\nn
%R_{T,I}^{\pi_d}(\Fc) &\le& 2C \Sigma_{t\notin\Ac(T),t\leq T}\Sigma_{j=1}^{M}\Sigma_{n=1}^{N}\exp(-a\delta^2w\log t) \\\nn
%&\le& 2NMC\Sigma_{t=1}^{\infty}t^{-a\delta^2w}\\
%&\le& 2NMC(1+\frac{1}{a\delta^2w-1})
%\end{eqnarray}}
% Where $C$ is the loss we have by doing a mistake.
% An upper bound on $C$ is $\Sigma_{j=1}^{M}\theta_{\sigma(j)}$
\end{proof}

\subsection{Combinatorial MAB with Dependent Arms}

In the classical MAB formulation, arms are independent. Reward observations from one arm do not
provide information about the quality of other arms. As a result, regret grows linearly with the number of arms.
However, many network optimization problems (such as optimal activation for online detection, shortest-path,  minimum spanning,
and dominating set under unknown and time-varying edge weights) lead to MAB with a large number of arms (e.g., the number of paths)
dependent through a small number of unknowns (e.g., the number of edges). While the dependency across arms can be ignored in
learning and existing learning algorithms directly apply, such a naive approach often yields a regret growing exponentially
with the problem size.

In~\cite{Liu&Zhao:WiOpt12}, we have shown that the DSEE approach can be extended to such combinatorial MAB problems to achieve
a regret that grows polynomially (rather than exponentially) with the problem size while maintaining its optimal logarithmic order
with time. The basic idea is to construct a set of properly chosen basis functions of the underlying network and explore only
the basis in the exploration sequence. The detailed regret analysis is rather involved
and is given in a separate study~\cite{Liu&Zhao:WiOpt12}.

\subsection{Restless MAB under Unknown Markov Dynamics}

The classical MAB formulation assumes an i.i.d. or a rested Markov reward model (see, e.g., \cite{Anantharam:87-2,Tekin:10})
which applies only to systems without memory or
systems whose dynamics only result from the player's action. In many practical applications such as target tracking and
scheduling in queueing and communication networks,  the system has memory and continues
to evolve even when it is not engaged by the player. For example, channels continue to evolve even when they are not sensed; queues continue to
grow due to new arrivals even when they are not served; targets continue to move even when they are not monitored.

Such applications can be formulated as restless MAB where the state of each arm continues to
evolve (with memory) even when it is not played. More specifically, the state of
each arm changes according to an unknown Markovian transition rule
when the arm is played and according to an arbitrary unknown random
process when the arm is not played. In~\cite{Liu&etal:13IT}, we have extended the DSEE approach
to the restless MAB under both the centralized
(or equivalently, the single-player) setting and the decentralized
setting with multiple distributed players. We have shown that the DSEE approach offers
a logarithmic order of the so-called weak regret. The detailed derivation is rather involved
and is given in a separate study~\cite{Liu&etal:13IT}.

\section{Conclusion}
The DSEE approach addresses the fundamental tradeoff between
exploration and exploitation in MAB by separating, in time, the two
often conflicting objectives. It has a clearly defined tunable
parameter---the cardinality of the exploration sequence----which can
be adjusted to handle any reward distributions and the
lack of any prior knowledge on the reward models. Furthermore, the
deterministic separation of exploration from exploitation allows
easy extensions to variations of MAB, including decentralized MAB with multiple players and incomplete
reward observations under collisions, MAB with unknown Markov dynamics, and combinatorial MAB with dependent
arms that often arise in
network optimization problems such as the shortest path, the minimum spanning, and the dominating set problems
under unknown random weights.

In algorithm design, there is often a tension between performance and generality.
The generality of the DSEE approach comes at a price of finite-time performance.
Even though DSEE offers the optimal regret order for any distribution, simulations show that the leading constant in the regret
offered by DSEE is often inferior to that of classic policies proposed in~\cite{Lai&Robbins85AAM,AgrawalEtal95AAP,Auer&etal02ML}
that target at specific types of distributions.
Whether one can improve the finite-time performance of DSEE without scarifying
its generality is an interesting future research direction.

\end{document}